\documentclass[12pt,reqno]{amsproc}%
\usepackage{amsfonts}
\usepackage{amsmath}
\usepackage{amssymb}
\usepackage{graphicx}%
 \theoremstyle{plain}
\newtheorem{theorem}{Theorem}[section]

\newtheorem{corollary}[theorem]{Corollary}

\newtheorem{definition}[theorem]{Definition}

\newtheorem{lemma}[theorem]{Lemma}

\newtheorem{proposition}[theorem]{Proposition}
\newtheorem{remark}[theorem]{Remark}

\numberwithin{equation}  {section}

\begin{document}

\title{Joint spectrum shrinking maps on projections}

\author{Wenhua Qian$^1$}
\address{1. School of Mathematical Sciences, Chongqing Normal University, Chongqing, China}

\author{Dandan Xiao$^1$}
\address{ Email: whqian86@163.com, 1258940058@qq.com, 1482408506@qq.com, wuwm@amss.ac.cn,}

\author{Tanghong Tao$^1$}

\address{1648024603@qq.com}

\author{Wenming Wu$^{1,\dag}$}

\author{Xin Yi$^1$}

\thanks{Qian is supported in part by the Natural Science Foundation of Chongqing Science and Technology Commission(cstc2020jcyj-msxmX0723) and
the Research Foundation of Chongqing Educational Committee (No. KJQN2021000529). Wu($^\dag$ corresponding author) is supported in part by NSF of China (No. 11871127, No. 11971463).}

\begin{abstract}
Let $\mathcal H$ be a finite dimensional complex Hilbert space with dimension $n \ge 3$ and $\mathcal P(\mathcal H)$ the set of projections on $\mathcal H$.
Let $\varphi: \mathcal P(\mathcal H) \to \mathcal P(\mathcal H)$ be a surjective map. We show that $\varphi$ shrinks the joint spectrum of any two projections
if and only if it is joint spectrum preserving for any two projections and thus is induced by a ring automorphism on $\mathbb C$ in a particular way.
In addition, for an arbitrary $k \ge 3$, $\varphi$ shrinks the joint spectrum of any $k$ projections if and only if it is induced by a unitary or an anti-unitary.
Assume that $\phi$ is a surjective map on the Grassmann space of rank one projections. We show that $\phi$ is joint spectrum preserving for any $n$ rank one projections
if and only if it can be extended to a surjective map on $\mathcal P(\mathcal{H})$ which is spectrum preserving for any two projections.
Moreover, for any $k >n$, $\phi$ is joint spectrum shrinking for any $k$ rank one projections if and only if it is induced by a unitary or an anti-unitary.
\end{abstract}



\maketitle

\section{Introduction}
The well-known Gleason-Kahane-\.{Z}elazko theorem (\cite{G, HZ}) states that a nonzero linear functional $\rho: \mathcal A \to \mathbb C$ on a unital complex Banach algebra $\mathcal A$
is an algebra homomorphism if and only if $\rho$ maps every element inside its spectrum. It is easy to verify that a nonzero linear functional $\rho$ on $\mathcal A$ is an algebra homomorphism
if and only if $\rho$ is a Jordan homomorphism, that is, $\rho(I) = 1$ where $I$ is the unit of $\mathcal{A}$ and $\rho$ preserves the squares. Motivated by this classical result,
in \cite{Ka2} Kaplansky asked whether a unital linear map $\varphi: \mathcal A \to \mathcal B$ between unital complex Banach algebras
which shrinks spectrum (i.e., $\sigma(\varphi(A)) \subseteq \sigma(A), \ \forall \ A \in \mathcal A$) is a Jordan homomorphism.
Notice that a unital linear map $\varphi: \mathcal A \to \mathcal B$ is spectrum shrinking if and only if it is invertibility preserving.

It is well-known that in general Kaplansky problem has a negative answer. A counterexample can be found in \cite{Au1}.
A lot of work has been done on Kaplansky Problem by additional assumptions (see \cite{Au3, HS} for some survey). Aupetit conjectured that Kaplansky Problem has a positive answer
when both Banach algebras are semi-simple and the map $\varphi$ is surjective and he confirmed this conjecture for von Neumann algebras \cite{Au2}. This problem is still open,
even for C$^*$-algebras \cite{BS, HK}. It was proved in \cite{CHNR} that the conjecture is true for C$^*$-algebras if in addition $\varphi$ is positive.
In particular, some related maps on matrix algebras are also considered \cite{FS, To}.

Recall that \cite{Ya} the joint spectrum of a tuple of operators $A_1,A_2,\ldots,A_l$ acting on a Hilbert space $\mathcal{H}$ is the set
 $$\sigma([A_1, \dots, A_l])=\{(c_1, \dots, c_l) \in \mathbb{C}^l: c_1A_1+\dots+c_lA_l ~\text{is not invertible in}~ \mathcal B(\mathcal H) \}.$$ It is an interesting issue to discuss the mapping which shrinks or preserves the joint spectrum of operators. It is easy to verify that a unital map $\varphi: \mathcal A \to \mathcal B$ is spectrum shrinking if and only if it shrinks the joint spectrum of the $2$-tuple $[I, A]$
for any element $A \in \mathcal A$. Therefore according to Aupetit's results \cite{Au2}, we can obtain the form of the mapping preserving the joint spectrum of any two operators in $\mathcal{B}(\mathcal{H})$.


In this paper we will characterize the mappings which shrink or preserve the joint spectrum of a tuple of projections.

Assume that $\mathcal H$ is a finite dimensional Hilbert space.
We first consider a surjective map $\varphi$ on the set $\mathcal P(\mathcal H)$ of projections on $\mathcal H$ which shrinks the joint spectrum of any two projections.
We first show that $\varphi$ leaves every Grassmann space invariant. By showing that the restriction of $\varphi$ on each Grassmann space is bijective,
we get that $\varphi$ is bijective. A mathematical induction gives that $\varphi$ is determined by its action on rank $n-1$ projections and as a consequence
we obtain that $\varphi$ is a lattice isomorphism which preserves the joint spectrum of any two projections. If $n =2$, it is easy to verify that
a surjective map $\varphi: \mathcal P(\mathcal H) \to \mathcal P(\mathcal H)$ is joint spectrum preserving for any two projections if and only if
$\varphi$ is bijective with $\varphi(I)=I, \varphi(0)=0$. If $n \ge 3$, some further calculations in Section 2 give the following result.

\vspace{0.2cm}
\begin{theorem} \label{main thm 1}
 Assume that $3 \le n (=\dim(\mathcal H)) < + \infty$ and $\varphi: \mathcal P(\mathcal H) \to \mathcal P(\mathcal H)$ is a surjective map.
Then the followings are equivalent.`
  \begin{enumerate}
 \item[(1)] $\varphi$ shrinks the joint spectrum of any two projections;
 \item[(2)] $\varphi$ preserves the joint spectrum of any two projections;
 \item[(3)] there exist an orthonormal basis $\alpha_1, \alpha_2, \dots, \alpha_n$, a basis $\beta_1, \beta_2, \dots, \beta_n$ for $\mathcal H$ and a ring automorphism
 $f$ of $\mathbb C$ such that
$$\varphi(P)(\mathcal H) = \{\hat{f}(P \xi): \xi \in \mathcal H\}, \forall \ P \in \mathcal P(\mathcal H),$$
where $\hat{f}: \mathcal H \to \mathcal H$ is induced by $f$ with
$$\hat{f}(z_1 \alpha_1 + z_2 \alpha_2 + \dots + z_n \alpha_n) = f(z_1) \beta_1 + f(z_2) \beta_2 + \dots + f(z_n)\beta_n.$$
\end{enumerate}
\end{theorem}
\vspace{0.2cm}

Moreover, we consider a surjective map $\varphi$ on the set $\mathcal P(\mathcal H)$
which shrinks the joint spectrum of any $k(\geq 3)$ projections.
We will further prove that $\varphi$ preserves the orthogonality of projections (i.e., $PQ=0$ if and only if $\varphi(P)\varphi(Q)=0$) and obtain the following equivalent characterizations.

\vspace{0.2cm}
\begin{theorem} \label{thm 3.3}
Assume that $3 \le n (=\dim(\mathcal H)) < + \infty$ and $\varphi: \mathcal P(\mathcal H) \to \mathcal P(\mathcal H)$ is a surjective map.
Then the followings are equivalent.
\begin{enumerate}
\item[(1)] there exists $k_0 \ge 3$ such that $\varphi$ shrinks the joint spectrum of any $k_0$ projections;
\item[(2)] there exists $k_0 \ge 3$ such that $\varphi$ preserves the joint spectrum of any $k_0$ projections;
\item[(3)] for any $k \ge 3$, $\varphi$ shrinks the joint spectrum of any $k$ projections;
\item[(4)] for any $k \ge 3$, $\varphi$ preserves the joint spectrum of any $k$ projections;
\item[(5)] there exists a unitary or anti-unitary $U$ such that $\varphi(P) = U^*PU, \ \forall \ P \in \mathcal P(\mathcal H)$.
\end{enumerate}
\end{theorem}

\vspace{0.2cm}

We also investigate a surjective map $\phi$ on the set $\mathcal P_1(\mathcal H)$ of rank one projections which preserves the joint spectrum of any $n$ rank one projections.
 It is shown that $\phi$ is order-preserving in the sense that $\phi(P) \le \phi(P_1) \vee \phi(P_2) \vee \dots \vee \phi(P_k)$ if and only if $P \le P_1 \vee P_2 \vee \dots \vee P_k$
for any positive integer $k$ and any $P, P_1, \dots, P_k \in \mathcal P_1(\mathcal H)$. It follows that $\phi$ can be extended  to a surjective map on the set $\mathcal P(\mathcal H)$
of projections on $\mathcal H$ which preserves the joint spectrum of any two projections and we obtain the following result.

\vspace{0.2cm}

\begin{theorem} \label{main thm 2}
Assume that $3 \le n (=\dim(\mathcal H)) < + \infty$ and $\phi: \mathcal P_1 (\mathcal H) \to \mathcal P_1 (\mathcal H)$ is a surjective map.
Then the followings are equivalent.
 \begin{enumerate}
\item[(1)] $\phi$ preserves the joint spectrum of any $n$ rank one projections;
\item[(2)] there exist an orthonormal basis $\alpha_1, \alpha_2, \dots, \alpha_n$, a basis $\beta_1, \beta_2, \dots, \beta_n$ for $\mathcal H$ and a ring automorphism $f$
of $\mathbb C$ such that
$$\phi(P)(\mathcal H) = \{\hat{f}(P \xi): \xi \in \mathcal H\}, \forall \ P \in \mathcal P_1(\mathcal H),$$
where $\hat{f}: \mathcal H \to \mathcal H$ is induced by $f$ as in Theorem \ref{main thm 1}.
\end{enumerate}
\end{theorem}
\vspace{0.2cm}

Moreover, if $\phi:  \mathcal P_1(\mathcal H) \to \mathcal P_1(\mathcal H)$ is surjective and shrinks the joint spectrum of any $n+1$ projections, we can show that
$\phi$ preserves the orthogonality of projections and obtain the following theorem.

\vspace{0.2cm}

\begin{theorem}\label{thm 4.14}
Assume that $3 \le n (=\dim(\mathcal H)) < + \infty$ and $\phi: \mathcal P_1(\mathcal H) \to \mathcal P_1(\mathcal H)$ is a surjective map.
Then the followings are equivalent.
\begin{enumerate}
\item[(1)] there exists $k_0 \ge n+1$ such that $\phi$ shrinks the joint spectrum of any $k_0$ projections;
\item[(2)] there exists $k_0 \ge n+1$ such that $\phi$ preserves the joint spectrum of any $k_0$ projections;
\item[(3)] for any $k \ge n+1$, $\phi$ shrinks the joint spectrum of any $k$ projections;
\item[(4)] for any $k \ge n+1$, $\phi$ preserves the joint spectrum of any $k$ projections;
\item[(5)] there exist a unitary or anti-unitary $U$ such that $\phi(P) = U^*PU, \ \forall \ P \in \mathcal P_1(\mathcal H)$.
\end{enumerate}
\end{theorem}

\section{Maps shrinking the joint spectrum of any two projections}
Let $\mathcal H$ be a Hilbert space with dimension $n < +\infty$. Denote by $\mathcal P(\mathcal H)$ and $\mathcal{P}_r(\mathcal H)$ (i.e., the order $r$ Grassmann space) the set of projections and the set of rank $r$ projections on $\mathcal H$.
In this section we assume that $\varphi : \mathcal P(\mathcal H) \to \mathcal P(\mathcal H)$ is a surjective map which shrinks the joint spectrum of any two projections,
i.e., $\sigma([\varphi(P), \varphi(Q)]) \subseteq \sigma([P,Q]), \ \forall \ P, Q \in \mathcal P(\mathcal H)$.

\begin{lemma} \label{lem 2.1}
$\varphi(I)= I, \varphi(0)=0$.
\end{lemma}
\begin{proof}
For any $Q \in \mathcal P(\mathcal H)$, $(1, 0) \notin \sigma([I, Q])$. Hence $(1, 0) \notin \sigma([\varphi(I), \varphi(Q)])$ and by the surjection of $\varphi$ we have $\varphi(I) = I$.
Since $(1, -1) \notin \sigma([I, 0])$, we have $(1, -1) \notin \sigma([\varphi(I), \varphi(0)]) = \sigma([I, \varphi(0)])$. Hence $\varphi(0) = 0$.
\end{proof}

For any $P, Q \in \mathcal P(\mathcal H)$, it is easy to verify that $P \vee Q = I$ if and only if $(1, 1) \notin \sigma([P, Q])$. Thus the following lemma is obvious.

\begin{lemma} \label{lem 2.2}
Let $P, Q \in \mathcal P(\mathcal H)$. If $P \vee Q =I$, then $\varphi(P) \vee \varphi(Q) =I$.
\end{lemma}

\begin{lemma} \label{lem 2.3}
Let $P, Q \in \mathcal P(\mathcal H)$. If $P \vee Q =I, P \wedge Q = 0$, then $\varphi(P) \vee \varphi(Q) =I, \varphi(P) \wedge \varphi(Q) = 0$.
\end{lemma}
\begin{proof}
If $P= I, Q=0$ or $P=0, Q=I$, then Lemma \ref{lem 2.1} gives the result. Assume that $P, Q \in \mathcal P(\mathcal H)\setminus \{0, I\}$. Then it follows from \cite{WJRQ} that
$(1, 1) \notin \sigma([P,Q]), (1, -1) \notin \sigma([P, Q])$. Hence $(1, 1) \notin \sigma([\varphi(P),\varphi(Q)]), (1, -1) \notin \sigma([\varphi(P),\varphi(Q)])$.
Thus $\varphi(P) + \varphi(Q)$ and $\varphi(P) - \varphi(Q)$ are both invertible. Hence  $\varphi(P) \vee \varphi(Q) =I, \varphi(P) \wedge \varphi(Q) = 0$.
\end{proof}

In the following, we denote by $r(P)$ the rank of $P$ for any $P \in \mathcal P(\mathcal H)$.

\begin{lemma} \label{lem 2.4}
Let $P, Q \in \mathcal P(\mathcal H)$. If $r(P) = r(Q)$, then $r(\varphi(P))=r(\varphi(Q))$.
Moreover, $\varphi(\mathcal P_k (\mathcal H)) = \mathcal P_k (\mathcal H),  \forall \ k \in \{0, 1, 2, \dots, n\}$.
\end{lemma}
\begin{proof}
Notice $\varphi(I)=I, \varphi(0)=0$. We may assume that $P, Q \in \mathcal P_k(\mathcal H)$, where $k \in \{1, 2, \dots, n-1\}$.

We first assume that $r(P \wedge Q) = k-1$. It follows that $r(P \vee Q)=k+1$. Then there exist linearly independent vectors $x_1, x_2, \dots, x_{k-1}, \alpha, \beta \in \mathcal H$
such that $P$ is the projection onto the subspace generated by $x_1, x_2, \dots, x_{k-1}, \alpha$ and $Q$ is the projection onto the subspace generated by
$x_1, x_2, \dots, x_{k-1}, \beta$. Take $R = P_1 + (I - P \vee Q)$, where $P_1$ is the rank one projection onto $\mathbb C (\alpha + \beta)$.
It follows that $P \vee R =I, P \wedge R = 0$ and $Q \vee R =I, Q \wedge R = 0$. By Lemma \ref{lem 2.3} we obtain that $\varphi(P) \vee \varphi(R) =I, \varphi(P) \wedge \varphi(R) = 0$
and $\varphi(Q) \vee \varphi(R) =I, \varphi(Q) \wedge \varphi(R) = 0$. Hence
$$r(\varphi(P))=n-r(\varphi(R))=r(\varphi(Q)).$$

Now assume that $r(P \wedge Q) = k-r$, where $1 \le r \le k$. Then there exist linearly independent vectors
$x_1, x_2, \dots, x_{k-r}, \alpha_1, \alpha_2, \dots, \alpha_r, \beta_1, \beta_2, \dots, \beta_r$ such that $P$ is the projection onto the subspace generated by
$x_1, x_2, \dots, x_{k-r}, \alpha_1, \alpha_2, \dots, \alpha_r$ and $Q$ is the projection onto the subspace generated by $x_1, x_2, \dots, x_{k-r}, \beta_1, \beta_2, \dots, \beta_r$.
Take $Q_0 = P, Q_r = Q$. For each $i \in \{1, 2, \dots, r-1\}$, let $Q_i$ be the projection onto the subspace generated by
$x_1, x_2, \dots, x_{k-r}, \beta_1, \dots, \beta_i, \alpha_{i+1}, \dots, \alpha_r$. It follows that $Q_0, Q_1, \dots, Q_r \in \mathcal P_k(H)$ and $r(Q_i \wedge Q_{i+1}) = k-1$
for every $i \in \{0, 1, \dots, r-1 \}$. Then the result of the previous paragraph implies that
$$r(\varphi(P))=r(\varphi(Q)).$$

Hence there exists a map $g : \{0,1,2, \dots, n\} \to \{0,1,2,\dots, n\}$ such that
$\varphi(\mathcal P_k(\mathcal H)) \subseteq \mathcal P_{g(k)}(\mathcal H), \forall \ k \in \{0, 1, 2, \dots, n\}$. By the fact that $\varphi$ is surjective,
we obtain that $g$ is a bijection and $\varphi(\mathcal P_k(\mathcal H)) = \mathcal P_{g(k)}(\mathcal H), \forall \ k \in \{0, 1, 2, \dots, n\}$.
In particular, by Lemma \ref{lem 2.3} we have $g(n-k)= n-g(k), \forall \ k \in \{0, 1, 2, \dots, n\}$.

Clearly, $g(0)=0$ and $g(n) =n$. Assume that $s = g(1) > 1$. Then $g(n-1) = n - s < n-1$. By the fact that $g$ is a bijection, there exists  $l >1$ such that $g(l)=1$.
Take two projections $P_1 \in \mathcal P_{n-1}(\mathcal H), P_2 \in \mathcal P_{l}(\mathcal H)$ with $P_1 \vee P_2 =I$.
It follows that $\varphi(P_1) \in \mathcal P_{n-s}(H), \varphi(P_2) \in \mathcal P_{1}(H)$. Then $r(\varphi(P_1))+r(\varphi(P_2)) < n$.
Therefore $\varphi(P_1) \vee \varphi(P_2) \neq I$ and we obtain a contradiction according to Lemma \ref{lem 2.2}. Hence $g(1)=1, g(n-1)=n-1$.
Continuing in this way, we have $\varphi(\mathcal P_k (\mathcal H)) = \mathcal P_k (\mathcal H),  \forall \ k \in \{0, 1, 2, \dots, n\}$.
\end{proof}

In the following we will show that the restriction of $\varphi$ on each Grassmann space $\mathcal P_k (\mathcal H)$ is a bijection and thus $\varphi$ is a bijection.
We present two necessary lemmas.

\begin{lemma} \label{lem 2.5}
Let $Q \in \mathcal P_{n-1}(\mathcal H), P \in \mathcal{P}(\mathcal H)$. If $\varphi(P) \le \varphi(Q)$, then $P \le Q$.
Moreover, $\varphi|_{\mathcal P_{n-1}(\mathcal H)}$ is a bijection.
\end{lemma}
\begin{proof}
By Lemma \ref{lem 2.4}, $\varphi(Q) \in \mathcal P_{n-1}(\mathcal H)$. Since  $\varphi(P) \le \varphi(Q)$, $\varphi(P) \vee \varphi(Q) \neq I$.
By Lemma \ref{lem 2.2} and the fact that $Q \in \mathcal P_{n-1}(H)$, we have $P \le Q$.  It is easy to verify that $\varphi|_{\mathcal P_{n-1}(\mathcal H)}$ is a bijection.
\end{proof}

For convenience, we denote by $\Phi = \phi|_{\mathcal P_{n-1}(\mathcal H)}$ in the following proposition.

\begin{proposition} \label{prop 2.7}
Let $k \in \{1, 2, \dots, n\}$ and $P \in \mathcal P_{n-k}(\mathcal H)$. Assume that $P' \in \mathcal P_{n-k}(\mathcal H)$ with $\varphi(P')=P$.
Then for any $k$ projections $Q_1, Q_2, \dots, Q_k \in \mathcal P_{n-1}(\mathcal H)$ with $P = Q_1 \wedge Q_2 \wedge \dots \wedge Q_k$,
$P' = \wedge_{1 \le i \le k} \Phi^{-1}(Q_i)$.
Moreover, $\varphi$ is a bijection.
\end{proposition}
\begin{proof}
We prove the result by a mathematical induction on $k$. From Lemma \ref{lem 2.5}, the result is true when $k=1$. Assume that the result is true when $k=s$.
Now let $k = s + 1$ and assume that $Q_1, Q_2, \dots, Q_s, Q_{s+1} \in \mathcal P_{n-1}(\mathcal H)$ with $P = Q_1 \wedge Q_2 \wedge \dots \wedge Q_s \wedge Q_{s+1}$.

Take $P_1 = Q_1 \wedge Q_2 \wedge \dots \wedge Q_s$ and $P_2 = Q_1 \wedge Q_2 \wedge \dots \wedge Q_{s-1} \wedge Q_{s+1}$.
Clearly $P_1, P_2$ are two different projections in $\mathcal P_{n-s}(\mathcal H)$. By the assumption that the result is true when $k=s$,
we have $\varphi|_{\mathcal P_{n-s}}(\mathcal H)$ is a bijection and
\begin{eqnarray}
&&(\varphi|_{\mathcal P_{n-s}(\mathcal H)})^{-1}(P_1) = \Phi^{-1}(Q_1) \wedge \Phi^{-1}(Q_2) \wedge \dots \wedge \Phi^{-1}(Q_s), \notag \\
&&(\varphi|_{\mathcal P_{n-s}(\mathcal H)})^{-1}(P_2) = \Phi^{-1}(Q_1) \wedge \Phi^{-1}(Q_2) \wedge \dots \wedge \Phi^{-1}(Q_{s-1}) \wedge \Phi^{-1}(Q_s+1). \label{2.7.1}
\end{eqnarray}
By Lemma \ref{lem 2.5}, $P' \le \Phi^{-1}(Q_i)$ for each $i \in \{1, 2, \dots, s+1\}$.
Hence $P' \le (\varphi|_{\mathcal P_{n-s}(\mathcal H)})^{-1}(P_1) \wedge (\varphi|_{\mathcal P_{n-s}(\mathcal H)})^{-1}(P_2)$.
Since $\varphi|_{\mathcal P_{n-s}}$ is a bijection, $(\varphi|_{\mathcal P_{n-s}(\mathcal H)})^{-1}(P_1) \neq (\varphi|_{\mathcal P_{n-s}(\mathcal H)})^{-1}(P_2)$ and
thus $r((\varphi|_{\mathcal P_{n-s}(\mathcal H)})^{-1}(P_1) \wedge (\varphi|_{\mathcal P_{n-s}(\mathcal H)})^{-1}(P_2)) \le n-s-1 = r(P')$.
Therefore $P' = (\varphi|_{\mathcal P_{n-s}(\mathcal H)})^{-1}(P_1) \wedge (\varphi|_{\mathcal P_{n-s}(\mathcal H)})^{-1}(P_2)=
\wedge_{1 \le i \le s+1} \Phi^{-1}(Q_i)$ from (\ref{2.7.1}). Hence $\varphi|_{\mathcal P_{n-s-1}}(\mathcal H)$ is also a bijection. Moreover, $\varphi$ is a bijection.
\end{proof}

According to Proposition \ref{prop 2.7}, we have the following corollary.

\begin{corollary} \label{cor 2.8}
If $P, Q \in \mathcal P(H)$, then we have the following results.
\begin{itemize}
\item[(1)] If $P \ne I$ and $\{Q_{\lambda}: \lambda \in \Omega \} \subseteq  \mathcal P_{n-1}(\mathcal H)$ with $ \wedge_{\lambda \in \Lambda} Q_{\lambda} = P$,
then $\varphi^{-1}(P) = \wedge_{\lambda \in \Omega} \varphi^{-1}(Q_{\lambda})$;
\item[(2)] If $P \le Q$, then $\varphi^{-1}(P) \le \varphi^{-1}(Q)$;
\item[(3)] $\varphi^{-1}(P \wedge Q) = \varphi^{-1}(P) \wedge \varphi^{-1}(Q)$, $\varphi^{-1}(P \vee Q) = \varphi^{-1}(P) \vee \varphi^{-1}(Q)$;
\end{itemize}
\end{corollary}
\begin{proof}
(1) Notice that $P \le Q_{\lambda}$ for every $\lambda \in \Omega$. It follows from Lemma \ref{lem 2.5} that
$\varphi^{-1}(P) \le \varphi^{-1}(Q_{\lambda}), \forall \ \lambda \in \Omega$ and hence $\varphi^{-1}(P) \le \wedge_{\lambda \in \Omega} \varphi^{-1}(Q_{\lambda})$.
Assume that $P \in \mathcal P_{n-k}(\mathcal H)$. Then there exist $\lambda_1, \lambda_2, \dots, \lambda_k \in \Omega$ such that
$P = \wedge_{1 \le i \le k} Q_{\lambda_i}$. It follows from Proposition \ref{prop 2.7} that $\varphi^{-1}(P) =  \wedge_{1 \le i \le k} \varphi^{-1}(Q_{\lambda_i})$. Hence
$$\varphi^{-1}(P) = \wedge_{\lambda \in \Omega} \varphi^{-1}(Q_{\lambda}).$$

(2) This is clear from (1).

(3) From (1), it is obvious that $\varphi^{-1}(P \wedge Q) = \varphi^{-1}(P) \wedge \varphi^{-1}(Q)$. By (2), we have $\varphi^{-1}(P) \vee \varphi^{-1}(Q) \le \varphi^{-1}(P \vee Q)$.
Since $r(P) = r(\varphi^{-1}(P)), r(Q) = r(\varphi^{-1}(Q))$ and $r(\varphi^{-1}(P) \wedge \varphi^{-1}(Q)) = r(\varphi^{-1}(P \wedge Q))=r(P \wedge Q)$,
we have $r(\varphi^{-1}(P) \vee \varphi^{-1}(Q)) = r(\varphi^{-1}(P \vee Q))$ and therefore $\varphi^{-1}(P \vee Q) = \varphi^{-1}(P) \vee \varphi^{-1}(Q)$.
\end{proof}

Now we are able to prove the following result, which states that a surjective map $\varphi: \mathcal P(\mathcal H) \to \mathcal P(\mathcal H)$
that shrinks the joint spectrum of any two projections will preserve the joint spectrum of any two projections.

\begin{theorem} \label{thm 2.9}
If $\varphi : \mathcal P(\mathcal H) \to \mathcal P(\mathcal H)$ is a surjective map such that
$\sigma([\varphi(P), \varphi(Q)]) \subseteq \sigma([P,Q]), \ \forall \ P, Q \in \mathcal P(\mathcal H)$, then
$$\sigma([\varphi(P), \varphi(Q)]) = \sigma([P, Q]), \ \forall \ P, Q \in \mathcal P(\mathcal H).$$
\end{theorem}
\begin{proof}
Note that $\varphi(I)=I, \varphi(0) = 0$. By Corollary \ref{cor 2.8}, we obtain that $P \vee Q = I$ if and only if $\varphi(P) \vee \varphi(Q)=I$ and
$P \wedge Q=0$ if and only if $\varphi(P) \wedge \varphi(Q) =0$. It follows from the result in \cite{WJRQ} that $\sigma([\varphi(P), \varphi(Q)]) = \sigma([P, Q])$
for any $P, Q \in \mathcal P(\mathcal H)$.
\end{proof}

\begin{remark} \label{rem 2.10}
By the previous discussions, we obtain that $\varphi$ is bijective and $\varphi^{-1}$ also preserves the joint spectrum of any two projections in $\mathcal P(\mathcal H)$.
Applying the same argument to $\varphi = (\varphi^{-1})^{-1}$, it follows from Corollary \ref{cor 2.8} that $\varphi(P) \vee \varphi(Q) = \varphi(P \vee Q),
\varphi(P) \wedge \varphi(Q) = \varphi(P \wedge Q)$ and $\varphi(P) \le \varphi(Q)$ if and only if $P \le Q$ for every $P, Q \in \mathcal P(H)$.
If $\mathcal H$ has dimension $2$, then any two different rank one projections $P, Q$ satisfy $P \vee Q =I, P \wedge Q = 0$ and therefore
$\sigma([P, Q]) = \{(c, 0): c \in \mathbb{C} \} \cup \{ (0, c): c \in \mathbb{C} \}$. Hence any bijection on $\mathcal P(\mathcal H)$ which fixes $0$ and $I$
preserves the joint spectrum of any two projections. In the following of this section we focus on the case that $3 \le n < +\infty$.
Fix an orthonormal basis $\alpha_1, \alpha_2, \dots, \alpha_n$ for $\mathcal H$. For a nonzero vector $\alpha \in \mathcal H$, we denote by $P_{\alpha}$ the projection onto
the one-dimensional subspace generated by $\alpha$. It follows that $\varphi(P_{\alpha_i})$ is a rank one projection for each $i \in \{1, 2, \dots, n\}$ and there exist nonzero vectors
$\beta_1, \beta_2, \dots, \beta_n$ such that $P_{\beta_{i}} = \varphi(P_{\alpha_i})$. By the fact that $\varphi(P) \vee \varphi(Q) = \varphi(P \vee Q)$
for every $P, Q \in \mathcal P(H)$, we have $\vee_{i}P_{\beta_i} = \varphi(\vee_{i}P_{\alpha_i}) = I$ and therefore $\beta_1, \beta_2, \dots, \beta_n$ are linearly independent.
Notice that for every $i \in \{2, \dots, n\}$, $P_{\alpha_1 + \alpha_i}$ is a rank one projection and $P_{\alpha_1 +  \alpha_i} \le P_{\alpha_1} \vee P_{\alpha_i}$.
We have $\varphi(P_{\alpha_1 + \alpha_i}) \le P_{\beta_1} \vee P_{\beta_i}$
and thus there exists a constant $0 \ne t_i \in \mathbb{C}$ such that $\varphi(P_{\alpha_1 + \alpha_i}) = P_{\beta_1 + t_i \beta_i}$.
\end{remark}

\begin{lemma} \label{lem 2.11}
With notations as in Remark \ref{rem 2.10}, we have that
$$\varphi(P_{\alpha_i + \alpha_j}) = P_{t_i \beta_i + t_j \beta_j}, \ \forall \ i,j \in \{2, 3, \dots, n\}.$$
\end{lemma}
\begin{proof}
Fix $i, j \in \{2, 3,\dots, n\}$. We only need to consider the case when $i \ne j$. Let $\alpha = \alpha_1 + \alpha_i + \alpha_j$.
It follows that  $P_\alpha \le P_{\alpha_1 + \alpha_i} \vee P_{\alpha_j}, P_\alpha \le P_{\alpha_1 + \alpha_j} \vee P_{\alpha_i}$. Therefore
$$\varphi(P_\alpha) \le P_{\beta_1 + t_i \beta_i} \vee P_{\beta_j}, \varphi(P_\alpha) \le P_{\beta_1 + t_j \beta_j} \vee P_{\beta_i}.$$
Hence there exists constants $c_1, c_2$ with $\beta_1 + t_1 \beta_i + c_{1} \beta_j,\beta_1 + t_j \beta_j + c_2 \beta_i$ are both in the range of $\varphi(P_{\alpha})$.
Since $\beta_1, \beta_2, \dots, \beta_n$ are linearly independent, we have $c_1 = t_j, c_2 = t_i$ and
\begin{align}
\varphi(P_{\alpha_1 + \alpha_i + \alpha_j}) = P_{\beta_1 + t_i \beta_i + t_j \beta_j}. \label{2.11.1}
\end{align}
 On the other hand, $\alpha = \alpha_1 + (\alpha_i + \alpha_j)$. It follows that
 \begin{align}
 \varphi(P_\alpha) \le \varphi(P_{\alpha_1}) \vee \varphi(P_{\alpha_i+\alpha_j}). \label{2.11.2}
 \end{align}
Clearly, $\varphi(P_{\alpha_i+\alpha_j}) \le P_{\beta_i} \vee P_{\beta_j}$. It follows from (\ref{2.11.1}) and (\ref{2.11.2}) that
$\varphi(P_{\alpha_i + \alpha_j}) = P_{t_i \beta_i + t_j \beta_j}$.
\end{proof}

By Remark \ref{rem 2.10} and Lemma \ref{lem 2.11}, replacing $\beta_i$ by $t_i \beta_i$ for every $i \in \{2,3, \dots, n\}$ if necessary, in the following of this section we may assume that
$$\varphi(P_{\alpha_i + \alpha_j}) = P_{\beta_i + \beta_j}, \ \forall \ i,j \in \{1,2, \dots, n\}.$$

\begin{remark} \label{rem def f}
Notice that for every $z \in \mathbb{C}$, $P_{\alpha_1 + z\alpha_2} \le P_{\alpha_1} \vee P_{\alpha_2}$.
Hence $\varphi(P_{\alpha_1 + z\alpha_2}) \le P_{\beta_1} \vee P_{\beta_2}$. Thus we can define a map $f: \mathbb C \to \mathbb C$ with
$\varphi(P_{\alpha_1 + z\alpha_2}) = P_{\beta_1 + f(z) \beta_2}$. It is easy to verify that $f$ is a bijection with $f(1)=1, f(0) =0$.
\end{remark}

\begin{lemma} \label{lem 2.13}
For each $i \in \{ 2,3, \dots, n \}$ and $z \in \mathbb{C}$, $\varphi(P_{\alpha_1 + z\alpha_i}) = P_{\beta_1 + f(z) \beta_i}$.
\end{lemma}
\begin{proof}
We only need to prove the result for $i > 2$. Let $\alpha = \alpha_1 + z\alpha_2 + z\alpha_i$. It follows that
$P_\alpha \le P_{\alpha_1} \vee P_{\alpha_2 + \alpha_i}$ and $P_\alpha \le P_{\alpha_1 + z\alpha_2} \vee P_{\alpha_i}$.
Therefore $\varphi(P_\alpha) \le P_{\beta_1} \vee P_{\beta_2 + \beta_i}, \varphi(P_\alpha) \le P_{\beta_1 + f(z)\beta_2} \vee P_{\beta_i}$ and there exist two constants $c_1, c_2$
such that $\beta_1 + c_1(\beta_2 + \beta_i)$ and $\beta_1 + f(z) \beta_2 + c_2 \beta_i$ are both in the range of $\varphi(P_\alpha)$. Hence $c_1 = c_2 = f(z)$ and
\begin{align}
\varphi(P_{\alpha}) = P_{\beta_1 + f(z) \beta_2 + f(z) \beta_i}. \label{2.13.1}
\end{align}
On the other hand, $\alpha = (\alpha_1 + z\alpha_i) + \alpha_2$ and thus $\varphi(P_{\alpha}) \le \varphi(P_{\alpha_1 + z\alpha_i}) \vee P_{\beta_2}$.
By the fact that $\varphi(P_{\alpha_1 + z\alpha_i}) \le P_{\beta_1} \vee P_{\beta_i}$ and (\ref{2.13.1}) we obtain that $\varphi(P_{\alpha_1 + z\alpha_i}) = P_{\beta_1 + f(z) \beta_i}$.
\end{proof}

\begin{lemma} \label{lem 2.16}
The map $f$ given in Remark \ref{rem def f} is a ring automorphism of $\mathbb{C}$.
That is, $f$ is a bijection and $f(z_1 + z_2) = f(z_1) + f(z_2), f(z_1 z_2) = f(z_1)f(z_2)$ for any $z_1, z_2 \in \mathbb C$.
\end{lemma}
\begin{proof}
Clearly $f$ is a bijection and $f(0)=0$. We only need to prove the result for the case that $z_1, z_2$ are both nonzero.

Take $\alpha = (\alpha_1 + (z_1+z_2)\alpha_2) + \alpha_3 = (\alpha_1+z_1\alpha_2)+(\alpha_3 + z_2 \alpha_2) = (\alpha_1 + \alpha_3) + (z_1+z_2) \alpha_2$.
Then $\varphi(P_\alpha) \le P_{\beta_1 + f(z_1 + z_2)\beta_2} \vee P_{\beta_3}$, $\varphi(P_\alpha) \le P_{\beta_1+f(z_1)\beta_2} \vee P_{\beta_3 + f(z_2) \beta_2}$ and $\varphi(P_\alpha) \le P_{\beta_1+\beta_3} \vee P_{\beta_2}$.
A similar calculation as in the proof of Lemma \ref{lem 2.11} gives that $f(z_1 + z_2) =f(z_1) + f(z_2)$.

Take $\xi = \alpha_1 + z_1z_2 \alpha_2 + z_1 \alpha_3$. It follows that $P_{\xi} \le P_{\alpha_1 + z_1z_2 \alpha_2} \vee P_{\alpha_3},
P_{\xi} \le P_{\alpha_1} \vee P_{\alpha_3+z_2 \alpha_2}$ and $P_{\xi} \le P_{\alpha_1 + z_1 \alpha_3} \vee P_{\alpha_2}$. Hence we have
$$\varphi(P_{\xi}) \le P_{\beta_1 + f(z_1z_2) \beta_2} \vee P_{\beta_3}, \varphi(P_{\xi}) \le P_{\beta_1} \vee P_{\beta_3 + f(z_2) \beta_2}, \varphi(P_{\xi}) \le P_{\beta_1 + f(z_1)\beta_3} \vee P_{\beta_2}.$$
It follows that there exist constants $c_1, c_2, c_3 \in \mathbb{C}$ such that
$$\beta_1 +  f(z_1z_2)\beta_2 + c_1 \beta_3, \beta_1 +  c_2(\beta_3 + f(z_2) \beta_2), \beta_1 + f(z_1)\beta_3 + c_3 \beta_2$$
are vectors in the range of $\varphi(P_{\xi})$. By the fact that $\beta_1, \beta_2, \beta_3$ are linearly independent, we have
$$\beta_1 +  f(z_1z_2)\beta_2 + c_1 \beta_3=\beta_1 +  c_2(\beta_3 + f(z_2) \beta_2)=\beta_1 + f(z_1)\beta_3 + c_3 \beta_2,$$
which implies that $f(z_1z_2) = f(z_1)f(z_2)$.
\end{proof}

\begin{corollary} \label{cor 2.14}
For any two distinguished numbers $i, j \in \{1,2,\dots, n\}$,
$$\varphi(P_{\alpha_i + z\alpha_j}) = P_{\beta_i + f(z) \beta_j}.$$
\end{corollary}
\begin{proof}
We may assume that $i \ne 1$ and $z \ne 0$. If $j =1$, then $\varphi(P_{\alpha_i + z\alpha_j}) = \varphi(P_{\alpha_1 +\frac{1}{z}\alpha_i}) = P_{\beta_1 + f(\frac{1}{z})\beta_i}$. By
Lemma \ref{lem 2.16} we obtain that $f(\frac{1}{z}) = \frac{1}{f(z)}$. Hence $P_{\beta_1 + f(\frac{1}{z})\beta_i} = P_{\beta_i+ z \beta_1}$.

Assume that $i, j$ are distinguished numbers in $\{ 2,3,\dots, n \}$. Clearly $\varphi(P_{\alpha_i + z\alpha_j}) \le P_{\beta_i} \vee P_{\beta_j}$. Take $\alpha = (\alpha_1 + z\alpha_j) + \alpha_i = (\alpha_1 + \alpha_i) + z\alpha_j = (\alpha_1 + \alpha_i + z\alpha_j)$.
Combining $\varphi(P_\alpha) \le P_{\beta_1 + f(z)\beta_j} \vee P_{\beta_i}, \varphi(P_\alpha) \le P_{\beta_1 + \beta_i} \vee P_{\beta_j}$ with
$\varphi(P_\alpha) \le P_{\beta_1} \vee \varphi(P_{\alpha_i+z\alpha_j})$, we obtain the required result.
\end{proof}

By Corollary \ref{cor 2.14}, a similar argument as in the proof of Lemma \ref{lem 2.11} and Lemma \ref{lem 2.13} yields the following result(a mathematical induction is needed).
We omit its proof.

\begin{lemma} \label{lem 2.15}
For any $z_2, z_3, \dots, z_n \in \mathbb C$ and any perturbation $i_1, i_2, \dots, i_n$ of $1, 2, \dots, n$,
$$\varphi(P_{\alpha_{i_{1}} + z_2\alpha_{i_2} + \dots +z_n \alpha_{i_{n}}} ) = P_{\beta_{i_1} + f(z_2) \beta_{i_2} + \dots + f(z_n) \beta_{i_{n}}}.$$
\end{lemma}

By Lemma \ref{lem 2.16} and Lemma \ref{lem 2.15}, we have the following corollary directly.

\begin{corollary} \label{cor 2.17}
 For any $z_1, z_2, z_3, \dots, z_n \in \mathbb C$ which are not all zero,
 $$\varphi(P_{z_1\alpha_1 + z_2\alpha_2 + \dots +z_n \alpha_n} ) = P_{f(z_1)\beta_1 + f(z_2) \beta_2 + \dots + f(z_n) \beta_n}.$$
\end{corollary}

\begin{definition} \label{def 2.18}
Let $\alpha_1, \alpha_2, \dots, \alpha_n$ be an orthonormal basis and $\beta_1, \beta_2, \dots, \beta_n$ a basis for $\mathcal H$. For any ring automorphism
$f: \mathbb C \to \mathbb C$, define a map $\hat{f}: \mathcal H \to \mathcal H$ (depending on $\alpha_1, \alpha_2, \dots, \alpha_n$ and $\beta_1, \beta_2, \dots, \beta_n$) as follows:
$$\hat{f}(z_1 \alpha_1+ z_2 \alpha_2+ \dots + z_n \alpha_n) = f(z_1)\beta_1 + f(z_2) \beta_2 + \dots + f(z_n) \beta_n.$$
\end{definition}

\begin{remark} \label{rem 2.19}
Given an orthonormal basis $\alpha_1, \alpha_2, \dots, \alpha_n$ and a basis $\beta_1, \beta_2, \dots, \beta_n$ for $\mathcal H$, we define a map
$\psi: \mathcal P(\mathcal H) \to \mathcal P(\mathcal H)$ by $\psi(P)(\mathcal H) = \{\hat{f}(P \xi): \xi \in \mathcal H\}$, where $\hat{f}$ is defined in Definition \ref{def 2.18}.
It is clear that $\psi$ is bijective with
$$\psi(P \vee Q) = \psi(P) \vee \psi(Q), \psi(P \wedge Q) = \psi(P) \wedge \psi(Q)$$
for any $P, Q \in \mathcal P(\mathcal H)$. Notice that for any positive integer $k$ and any matrix $ (a_{ij})_{k \times k} \in M_k (\mathbb C)$,
$ det((f(a_{ij})))  = f(det((a_{ij})))$. Therefore $det (f(a_{ij}))  \ne 0$ if and only if $det(a_{ij}) \ne 0$. Hence $r(\psi(P)) = r(P)$ for any $P \in \mathcal P(\mathcal H)$.
Therefore $\psi(P) \vee \psi(Q) = I$ if and only if $P \vee Q=I$ and, $\psi(P) \wedge \psi(Q)=0$ if and only if $P \wedge Q =0$.
We obtain that $\psi$ is joint spectrum preserving for any two projections.
\end{remark}

\vspace{0.2cm}

{\bf Proof of Theorem \ref{main thm 1}.} Theorem \ref{thm 2.9} gives that $(1)$ and $(2)$ are equivalent. $(3) \Rightarrow (2)$ follows directly from Remark \ref{rem 2.19}. Now we only need to show that $(2) \Rightarrow (3)$. Assume that $\varphi$ preserves the joint spectrum of any two projections. It follows from Corollary \ref{cor 2.17} and the definition of $\hat{f}$ that $\varphi(P) = \{\hat{f}(P \xi): \xi \in \mathcal H\}, \forall \ P \in \mathcal P_1(\mathcal H)$. By the fact that $\hat{f}(\eta_1 + \eta_2) = \hat{f}(\eta_1)+\hat{f}(\eta_2)$ for any $\eta_1, \eta_2 \in \mathcal H$ and $\varphi(P \vee Q) = \varphi(P) \vee \varphi(Q)$, we obtain that
$\varphi(P) = \{\hat{f}(P \xi): \xi \in \mathcal H\}, \forall \ P \in \mathcal P(\mathcal H)$.

\section{Maps shrinking the joint spectrum of more than two projections}
Assume that $n \ge 3$. In this section we look at a surjective map $\varphi$ on $\mathcal P(\mathcal H)$ which shrinks the joint spectrum of any $k$ projections,
where $k \ge 3$. It is easy to verify that $\varphi$ also shrinks the joint spectrum of any $2$ projections and thus $\varphi$ is also induced by a ring automorphism $f$ on $\mathbb C$ as in Theorem \ref{main thm 1}.

\begin{lemma} \label{lem 3.1}
Assume that $\varphi: \mathcal P(\mathcal H) \to \mathcal P(\mathcal H)$ is a surjective map which shrinks the joint spectrum of any $3$ projections.
Then $\varphi$ preserves the orthogonality, i.e., if $PQ=0$, then $\varphi(P) \varphi(Q)=0$. In particular, the vectors $\beta_1, \beta_2, \ldots, \beta_n$ in Theorem \ref{main thm 1} are mutually orthogonal in $\mathcal H$.
\end{lemma}
\begin{proof}
By way of contradiction, we assume that there exist two nonzero projections $P, Q$ on $\mathcal H$ with $PQ=0$ and $\varphi(P) \varphi(Q) \ne 0$.
By the arguments in Section 2 we have $\varphi$ preserves the order of projections.
Replacing $P$ by $I-Q$ if necessary, we may assume that $P + Q = I$ and $\varphi(P) \varphi(Q) \ne 0$. It follows that there exist a unit vector $ \xi \in \varphi(P)(\mathcal H)$
such that $\varphi(Q) \xi \ne 0$. Let $c = \Vert\varphi(Q) \xi \Vert >0$. It is easy to verify that
$$\Vert \xi + \varphi(Q) \xi \Vert^2 = 1+3 c^2, \langle \xi, \xi + \varphi(Q) \xi \rangle = 1+c^2.$$
Take a rank one projection $R$ with $\xi + \varphi(Q) \xi$ in its range. It follows that $R \xi = \frac{1+c^2}{1+3c^2} (\xi + \varphi(Q)\xi)$. Hence we obtain that
$(\varphi(P) + \varphi(Q) -\frac{1+3c^2}{1+c^2}R) \xi =0$
and thus $\varphi(P) + \varphi(Q) - \frac{1+3c^2}{1+c^2} R$ is not invertible. Therefore $(1, 1, -\frac{1+3c^2}{1+c^2}) \in \sigma([\varphi(P), \varphi(Q), R])$.
However, $\frac{1+3c^2}{1+c^2} > 1$, $P+Q- \frac{1+3c^2}{1+c^2} \varphi^{-1}(R) = I - \frac{1+3c^2}{1+c^2} \varphi^{-1}(R)$ is invertible and
$(1, 1, -\frac{1+3c^2}{1+c^2}) \notin \sigma([P, Q, \varphi^{-1}(R)])$. We obtain a contradiction.

It is easy to verify that $\beta_1, \beta_2, \ldots, \beta_n$ in Theorem \ref{main thm 1} are mutually orthogonal in  $\mathcal H$.
\end{proof}

\begin{remark}
Notice that for each $i\in\{2,\ldots,n\}$, the projections $P_{\alpha_1+\alpha_i}$ and $P_{\alpha_1-\alpha_i}$ are orthogonal. Thus $\varphi(P_{\alpha_1+\alpha_i})$ and $\varphi(P_{\alpha_1-\alpha_i})$ are orthogonal. Then we have
$\langle \beta_1+\beta_i,\beta_1-\beta_i\rangle=0$, which implies that $\|\beta_1\|=\|\beta_i\|$ for any $i\in\{2,\ldots,n\}$. Thus replacing each $\beta_i$ by $\frac{\beta_i}{\|\beta_1\|}$ for $i=1,2,\ldots,n$, we may assume that $\beta_1,\beta_2,\ldots,\beta_n$ form an orthonormal basis for $\mathcal{H}$.
\end{remark}

\begin{lemma}\label{lem 3.2}
Assume that $\varphi: \mathcal P(\mathcal H) \to \mathcal P(\mathcal H)$ is a surjective map which shrinks the joint spectrum of any $3$ projections. Then either
$$f(z)=z, \ \forall \ z \in \mathbb C$$
or
$$f(z) = \bar{z}, \ \forall \ z \in \mathbb C,$$
where $f$ is the map defined in Section 2.
\end{lemma}
\begin{proof}
Notice that $f(i)^2 = f(i^2) = f(-1) =-1$. We have either $f(i) =i$ or $f(i)=-i$. In the following we only need to show that $f(t) = t$ for all $t \in \mathbb R$.
We first assume that $t \ge 0$. Let $P_1, P_2$ be rank one projections such that $t \alpha_1 + \sqrt{t} \alpha_2 \in P_1(\mathcal H)$ and
$ \alpha_1 - \sqrt{t} \alpha_2 \in P_2(\mathcal H)$. By Theorem \ref{main thm 1}, we have $f(t)\beta_1 + f(\sqrt{t}) \beta_2 \in \varphi(P_1)(\mathcal H)$ and
$\beta_1 - f(\sqrt{t}) \beta_2 \in \varphi(P_2)(\mathcal H)$. Notice that $P_1 P_2 =0$. By Lemma \ref{lem 3.1},  $\varphi(P) \varphi(Q) =0$.
Hence $\langle f(t)\beta_1 + f(\sqrt{t}) \beta_2,  \beta_1 - f(\sqrt{t}) \beta_2  \rangle = 0$. Therefore
$$f(t) = \vert f(\sqrt{t})\vert^2 \ge 0.$$

Assume that $s_1 < t < s_2$, where $s_1, s_2$ are rational numbers. Note that it is easy to verify that $f(s) = s$ for every rational number  $s \in \mathbb R$. Then
$$s_2-f(t) = f(s_2)-f(t) = f(s_2-t) \ge 0, f(t)-s_1 = f(t)-f(s_1) = f(t-s_1) \ge 0$$
and therefore $s_1 \le f(t) \le s_2$. Hence we have $f(t) = t$ for all $t \in \mathbb R$ and we obtain that (1) if $f(i) = i$, then $f(z) = z$ for all $z \in \mathbb C$ and
(2) if $f(i)=-i$, then $f(z) = \bar{z}$ for all $z \in \mathbb C$.
\end{proof}

\vspace{0.2cm}

{\bf Proof of Theorem \ref{thm 3.3}.}
It is clear that $(5) \Rightarrow (4) \Rightarrow (3) \Rightarrow (1)$ and $(4) \Rightarrow (2) \Rightarrow (1)$. We only need to show that $(1) \Rightarrow (5)$.
It follows from $(1)$ that $\varphi$ shrinks the joint spectrum of any $3$ projections. Hence we may assume $\alpha_1, \alpha_2, \dots, \alpha_n$ and $\beta_1, \beta_2, \dots, \beta_n$ in
Theorem \ref{main thm 1}  are two orthonormal basis for $\mathcal H$. By Lemma \ref{lem 3.2} we have either
$$f(z)=z, \ \forall \ z \in \mathbb C$$
or
$$f(z) = \bar{z}, \ \forall \ z \in \mathbb C,$$
where $f$ is the map defined in Remark \ref{rem def f}. If $f(z) =z$ for all $z \in \mathbb C$, then we define a unitary $U$ by
$U(z_1 \beta_1 + z_2 \beta_2 + \dots + z_n \beta_n) = z_1 \alpha_1 + z_2 \alpha_2 + \dots + z_n \alpha_n$. If $f(z) =\bar{z}$ for all $z \in \mathbb C$, then we define an anti-unitary
$U$ by $U(z_1 \beta_1 + z_2 \beta_2 + \dots + z_n \beta_n) = \bar{z_1} \alpha_1 + \bar{z_2} \alpha_2 + \dots + \bar{z_n} \alpha_n$. It follows from Theorem \ref{main thm 1} that
$\varphi(P) = U^*PU$ in both cases.

\section{Joint spectrum shrinking maps on rank one projections}
Assume that $n \ge 3$. In this section we assume that $\phi: \mathcal P_1 (\mathcal H) \to \mathcal P_1 (\mathcal H)$ is a surjective map. It is easy to verify that for any
positive integer $m < n$, the joint spectrum of any $m$ rank one projections $P_1, P_2, \dots, P_m$ is $\mathbb C^m$. Therefore every map on $\mathcal P_1(\mathcal H)$ preserves
the joint spectrum of any $m$ rank one projections if $m <n$.

\subsection{Maps preserving the joint spectrum of any $n$ rank one projections}
We start with a description of the joint spectrum of $n$ rank one projections.
\begin{lemma} \label{lem 4.1}
Let $P_1, P_2, \dots, P_n \in \mathcal P_1(\mathcal H)$. Then
\begin{enumerate}
\item[(1)] if $P_1 \vee P_2 \vee \dots \vee P_n \ne I$ , then $\sigma([P_1, P_2, \dots, P_n]) = \mathbb C^n$;
\item[(2)] if $P_1 \vee P_2 \vee \dots \vee P_n = I$, then $\sigma([P_1, P_2, \dots, P_n]) = \{(c_1, c_2, \dots, c_n) \in \mathbb C^n: c_1c_2 \dots c_n =0\}$.
\end{enumerate}
\end{lemma}
\begin{proof}
If $P_1 \vee P_2 \vee \dots \vee P_n \ne I$, then the range of any linear combination of $P_1, P_2, \dots, P_n$ is contained in the range of $P_1 \vee P_2 \vee \dots \vee P_n$
and thus any linear combination of $P_1, P_2, \dots, P_n$ is not invertible. Therefore $\sigma([P_1, P_2, \dots, P_n]) = \mathbb C^n$.

On the other hand, assume that $P_1 \vee P_2 \vee \dots \vee P_n = I$ and $c_1 P_1 + c_2 P_2 + \dots + c_n P_n$ is not invertible. Then there exists a nonzero vector
$\beta \in \mathcal H$ such that $c_1 P_1 \beta + c_2 P_2 \beta + \dots + c_n P_n \beta =0$. Hence
$c_i P_i \beta = -c_1 P_1 \beta - \dots - c_{i-1}P_{i-1}\beta - c_{i+1}P_{i+1}\beta- \dots - c_nP_n \beta =0$. By the fact that $P_1 \vee P_2 \vee \dots \vee P_n = I$ we have
$P_i \wedge (P_1 \vee \dots \vee P_{i-1} \vee P_{i+1} \vee \dots \vee P_n) = 0$ for each $i \in \{1, 2, \dots, n\}$.
If $c_1 c_2 \dots c_n \ne 0$, then $P_1 \beta = P_2 \beta = \dots = P_n \beta =0$, which is a contradiction to that $P_1 \vee P_2 \vee \dots \vee P_n = I$ and $\beta \ne 0$.
Therefore $\sigma([P_1, P_2, \dots, P_n]) \subseteq \{(c_1, c_2, \dots, c_n) \in \mathbb C^n: c_1c_2 \dots c_n =0\}$. It is obvious that
$\{(c_1, c_2, \dots, c_n) \in \mathbb C^n: c_1c_2 \dots c_n =0\} \subseteq \sigma([P_1, P_2, \dots, P_n])$.
\end{proof}

In order to prove the main result of this subsection, we will extend $\phi$ to a bijective map on $\mathcal P(\mathcal H)$ which preserves the joint spectrum of any two projections.
We give some necessary lemmas.

\begin{lemma} \label{lem 4.2}
Assume that $\phi: \mathcal P_1(\mathcal H) \to \mathcal P_1(\mathcal H)$ is a surjective map which preserves the joint spectrum of any $n$ rank one projections.
Then for any $1 \le k \le n$
and $P_1, P_2, \dots, P_k \in \mathcal P_1(\mathcal H)$,
$r(P_1 \vee P_2 \vee \dots \vee P_k)=k$ if and only if $r(\phi(P_1) \vee \phi(P_2)\vee \dots \vee \phi(P_k))=k$. In particular,  $\phi$ is injective.
\end{lemma}
\begin{proof}
We first prove the ``only if" part. Assume that $r(P_1 \vee P_2 \vee \dots \vee P_k)=k$. Then there exist $n-k$ rank one projections $P_{k+1}, P_{k+2}, \dots, P_{n}$ with
$P_1 \vee P_2 \vee \dots \vee P_n = I$. Since $\phi$  preserves the joint spectrum of any $n$ rank one projections, it follows from Lemma \ref{lem 4.1} that
$r(\phi(P_1) \vee \phi(P_2)\vee \dots \vee \phi(P_n))=n$. Hence $r(\phi(P_1) \vee \phi(P_2)\vee \dots \vee \phi(P_k))=k$. If $P, Q$ are two distinguished rank one projections,
then $r(P \vee Q)=2$. It follows that $r(\phi(P) \vee \phi(Q)) =2$ and hence $\phi$ is injective. The other direction follows directly from the fact that $\phi^{-1}$ also preserves
the joint spectrum of any rank one projections.
\end{proof}

\begin{lemma} \label{lem 4.3}
Assume that $\phi: \mathcal P_1(\mathcal H) \to \mathcal P_1(\mathcal H)$ is a surjective map which preserves the joint spectrum of any $n$ rank one projections. Then for any
$1 \le k \le n$ and $P_1, P_2, \dots, P_k \in \mathcal P_1(\mathcal H)$, $r(\phi(P_1) \vee \phi(P_2) \dots \vee \phi(P_k)) = r(P_1 \vee P_2 \dots \vee P_k)$.
\end{lemma}
\begin{proof}
Denote $s = r(P_1 \vee P_2 \vee \dots \vee P_k)$. It follows easily from Lemma \ref{lem 4.2} that $\phi$ is invertible on $\mathcal P_{1}(H)$ and $\phi^{-1}$ also preserves the
joint spectrum of any tuple of rank one projections. Hence we only need to show that $r(\phi(P_1) \vee \phi(P_2) \vee \dots \vee \phi(P_k)) \ge s$. Since
$s = r(P_1 \vee P_2 \vee \dots \vee P_k)$, there exist $i_1, i_2, \dots, i_s \in \{1, 2, \dots, k\}$ such that $r(P_{i_1} \vee P_{i_2} \vee \dots \vee P_{i_s})=s$. By Lemma \ref{lem 4.2},
$r(\phi(P_{i_1})\vee \phi(P_{i_2}) \vee \dots \vee \phi(P_{i_s}))=s$. Therefore $r(\phi(P_1) \vee \phi(P_2) \vee \dots \vee \phi(P_k)) \ge s$.
\end{proof}

\begin{lemma} \label{lem 4.4}
Assume that $\phi: \mathcal P_1(\mathcal H) \to \mathcal P_1(\mathcal H)$ is a surjective map which preserves the joint spectrum of any $n$ rank one projections.
For any $P, Q_1, Q_2, \dots, Q_k$ in $\mathcal P_1(\mathcal H)$, $P \le Q_1 \vee Q_2 \vee \dots \vee Q_k$ if and only if $\phi(P) \le \phi(Q_1) \vee \phi(Q_2) \vee \dots \vee \phi(Q_k)$.
In particular, if $E_1, E_2, \dots, E_l; F_1, F_2, \dots F_m \in \mathcal P_1(\mathcal H)$, then $E_1 \vee E_2 \vee \dots \vee E_l = F_1 \vee F_2 \vee \dots \vee F_m$ if and only if
$\phi(E_1) \vee \phi(E_2) \vee \dots \vee \phi(E_l) = \phi(F_1) \vee \phi(F_2) \vee \dots \vee \phi(F_m)$.
\end{lemma}
\begin{proof}
Assume that $P \le Q_1 \vee Q_2 \vee \dots \vee Q_k$. Then $r(P \vee Q_1 \vee \dots \vee Q_k) = r(Q_1 \vee \dots \vee Q_k)$. It follows from Lemma \ref{lem 4.3} that
$$r(\phi(P) \vee \phi(Q_1) \vee \dots \vee \phi(Q_k)) = r(\phi(Q_1) \vee \dots \vee \phi(Q_k)),$$
which implies that $\phi(P) \le \phi(Q_1) \vee \phi(Q_2) \vee \dots \vee \phi(Q_k)$. The other direction is similar from Lemma \ref{lem 4.3}.

If $E_1 \vee E_2 \vee \dots \vee E_l = F_1 \vee F_2 \vee \dots \vee F_m$, then the previous argument implies that $\phi(E_i) \le \phi(F_1) \vee \phi(F_2) \vee \dots \vee \phi(F_m)$
for each $i \in \{1,2, \dots, m\}$ and $\phi(F_j) \le \phi(E_1) \vee \phi(E_2) \vee \dots \vee \phi(E_l)$ for each $j \in \{1, 2, \dots, m\}$. Hence
$\phi(E_1) \vee \phi(E_2) \vee \dots \vee \phi(E_l) = \phi(F_1) \vee \phi(F_2) \vee \dots \vee \phi(F_m)$. The other direction is also similar.
\end{proof}

{\bf Proof of Theorem \ref{main thm 2}.} $(2) \Rightarrow (1):$ We naturally extend $\phi$ to a bijective map $\psi: \mathcal P(\mathcal H) \to \mathcal P(\mathcal H)$ by
$$\psi(P)(\mathcal H) = \{\hat{f}(P \xi): \xi \in \mathcal H\}, \forall \ P \in \mathcal P(\mathcal H).$$
Clearly $\psi$ is bijective and it follows from Theorem \ref{main thm 1} that $\psi$ is joint spectrum preserving for any two projections. Hence by the arguments in Section 2 we have
that $\psi(P \vee Q) = \psi(P) \vee \psi(Q)$ for any $P, Q \in \mathcal P(\mathcal H)$. Clearly $\psi(I) = I$. It follows that for any $n$ rank one projections
$P_1, P_2, \dots, P_n \in \mathcal P(\mathcal H)$, $P_1 \vee P_2 \vee \dots \vee P_n = I$ if and only if $\psi(P_1) \vee \psi(P_2) \vee \dots \vee \psi(P_n)=I$. By Lemma \ref{lem 4.1},
$\phi$ preserves the joint spectrum of any $n$ rank one projections.

$(1) \Rightarrow (2):$ We define a map $\Psi: \mathcal P(\mathcal H) \to \mathcal P(\mathcal H)$ as follows: if $P = P_1 \vee P_2 \vee \dots \vee P_k$ with
$P_1, P_2, \dots, P_k \in \mathcal P_1(\mathcal H)$, then define $$\Psi(P) = \phi(P_1) \vee \phi(P_2) \vee \dots \vee \phi(P_k)$$
and let $\Psi(0)=0$. By Lemma \ref{lem 4.4}, $\Psi$ is well-defined and injective. Fix $R \in \mathcal P(\mathcal H)$ and let $s = r(R)$.
Then there exist $s$ rank one projections $R_1, R_2, \dots, R_s \in \mathcal P_1(\mathcal H)$ such that $R=R_1 \vee R_2 \vee \dots \vee R_s$.
Note that $\phi$ is bijective. Take $R' = \phi^{-1}(R_1) \vee \phi^{-1}(R_2) \vee \dots \vee \phi^{-1}(R_s)$ and it follows that $\Psi(R')=R$. Therefore $\Psi$ is bijective.
To show $\Psi$ is joint spectrum preserving for any two projections, we only need to prove (1) for any $P, Q \in \mathcal P(\mathcal H)$, $P \vee Q = I$ if and only if
$\Psi(P) \vee \Psi(Q) = I$ and (2) for any $P, Q \in \mathcal P(\mathcal H)$, $P \vee Q =I, P \wedge Q =0$ if and only if $\Psi(P) \vee \Psi(Q) =I,
\Psi(P) \wedge \Psi(Q) = 0$. By Lemma \ref{lem 4.2} and the definition of $\Psi$, the proof of (1) is obvious.
Now assume $P \vee Q =I, P \wedge Q =0$. It follows that $r(P)+r(Q)=n$ and $\Psi(P) \vee \Psi(Q) =I$. By Lemma \ref{lem 4.3} we have $r(\Psi(P))=r(P), r(\Psi(Q))=r(Q)$.
Hence $r(\Psi(P))+r(\Psi(Q))=n$. Therefore $\Psi(P) \wedge \Psi(Q)=0$. The rest follows directly from Theorem \ref{main thm 1}.

\subsection{Maps shrinking the joint spectrum of more than $n$ rank one projections}
Now we assume that $\phi: \mathcal P_1(\mathcal H) \to \mathcal P_1(\mathcal H)$ is a surjective map which shrinks the joint spectrum of $n+1$ projections.
Notice that $\phi$ also shrinks the joint spectrum of any $n$ rank one projections. A similar argument as in Lemma \ref{lem 4.2} gives that $\phi$ is bijective
 on $\mathcal P_1(\mathcal H)$. We follow a similar line as in Section 3 to show that $\phi$ preserves the orthogonality.

\begin{lemma} \label{lem 4.10}
Assume that $\phi: \mathcal P_1(\mathcal H) \to \mathcal P_1(\mathcal H)$ is a surjective map which shrinks the joint spectrum of any $n+1$ projections.
Then $\phi$ preserves the orthogonality.
\end{lemma}
\begin{proof}
By way of contradiction, assume that $P, Q \in \mathcal P_1(\mathcal H)$ such that $PQ=0$ and $\phi(P) \phi(Q) \ne 0$. Take a unit vector $\xi \in \phi(P) \mathcal H$
such that $\phi(Q) \xi \ne 0$. Take a rank one projection $R$ with $\xi + \phi(Q)\xi$ in its range. Let $c = \Vert \phi(Q) \xi \Vert > 0$.
It follows that
\begin{eqnarray}
(\phi(P) + \phi(Q) - \frac{1+3c^2}{1+c^2}R)\xi =0. \label{4.10.1}
\end{eqnarray}
Notice that $R \le \phi(P) \vee \phi(Q)$. We have that $Ran(\phi(P) + \phi(Q) - \frac{1+3c^2}{1+c^2}R) \le \phi(P) \vee \phi(Q)$,
where $Ran(\phi(P) + \phi(Q) - \frac{1+3c^2}{1+c^2}R)$ denotes the range projection of $\phi(P) + \phi(Q) - \frac{1+3c^2}{1+c^2}R$.
It follows from (\ref{4.10.1}) that $r(Ran(\phi(P) + \phi(Q) - \frac{1+3c^2}{1+c^2}R)) = 1$.  Take $P_3, P_4, \dots, P_n \in \mathcal P_1(\mathcal H)$ such that
$P + Q + P_3+ \dots + P_n =I$. Since $\frac{1+3c^2}{1+c^2}>1$, $P + Q - \frac{1+3c^2}{1+c^2}\phi^{-1}(R) + P_3+ \dots + P_n =I-\frac{1+3c^2}{1+c^2}\phi^{-1}(R)$ is invertible.
Since $r(Ran(\phi(P) + \phi(Q) - \frac{1+3c^2}{1+c^2}R)) = 1$, $r(Ran(\phi(P) + \phi(Q) - \frac{1+3c^2}{1+c^2}R) + \phi(P_3)+ \dots + \phi(P_n))) \le n-1$ and thus
 $\phi(P) + \phi(Q) - \frac{1+3c^2}{1+c^2}R + \phi(P_3)+ \dots + \phi(P_n)$ is not invertible. We obtain a contradiction.
\end{proof}

\begin{lemma} \label{lem 4.11}
Assume that $Q_1, Q_2, \dots, Q_k$ are mutually orthogonal projections in $\mathcal P_1 (\mathcal H)$ and $P \in \mathcal P_1(\mathcal H)$ with $P \le Q_1 + Q_2 + \dots + Q_k$.
Then $\phi(P) \le \phi(Q_1)+ \phi(Q_2) + \dots + \phi(Q_k)$. Moreover, if $R_1, R_2, \dots, R_k$ are mutually orthogonal with $R_1 + R_2 + \dots + R_k = Q_1 + Q_2 + \dots + Q_k$,
then $\phi(R_1)+ \phi(R_2) + \dots + \phi(R_k)=\phi(Q_1)+ \phi(Q_2) + \dots + \phi(Q_k)$.
\end{lemma}
\begin{proof}
Take $Q_{k+1}, Q_{k+2}, \dots, Q_n \in \mathcal P_1(\mathcal H)$ with $Q_1 + Q_2 + \dots + Q_n = I$. It follows from Lemma \ref{lem 4.10} that
$\phi(Q_1), \dots, \phi(Q_k), \phi(Q_{k+1}), \dots, \phi(Q_n)$ are mutually orthogonal with sum $I$. Notice that for each $i \in \{ k+1, k+2, \dots, n \}$, $PQ_i =0$.
By Lemma \ref{lem 4.10} again, $\phi(P) \phi(Q_i) = 0, \ \forall \ i \in \{ k+1, k+2, \dots, n\}$. Hence $\phi(P) \le \phi(Q_1)+ \phi(Q_2)+ \dots + \phi(Q_k)$.
If $R_1, R_2, \dots, R_k$ are mutually orthogonal with $R_1 + R_2 + \dots + R_k = Q_1 + Q_2 + \dots + Q_k$, then we get $\phi(R_i) \le \phi(Q_1)+ \phi(Q_2)+ \dots + \phi(Q_k)$ and
 $\phi(Q_i) \le \phi(R_1)+ \phi(R_2)+ \dots + \phi(R_k)$ for each $i \in \{ 1, 2, \dots, k \}$.
 Hence $\phi(R_1)+ \phi(R_2) + \dots + \phi(R_k)=\phi(Q_1)+ \phi(Q_2) + \dots + \phi(Q_k)$.
\end{proof}

Now we can get the main result of this subsection.

\vspace{0.2cm}

{\bf Proof of Theorem \ref{thm 4.14}.} It is clear that $(5) \Rightarrow (4) \Rightarrow (3) \Rightarrow (1)$ and $(4) \Rightarrow (2) \Rightarrow (1)$. In the following we only need to verify $(1) \Rightarrow (5)$.

We define a map $\Psi: \mathcal P(\mathcal H) \to \mathcal P(\mathcal H)$ as follows, $\Psi(0)=0$ and $\Psi(P)=\phi(P_1)+\phi(P_2)+\cdots +\phi(P_k)$ when $P \in \mathcal P(\mathcal H)$ and $P = P_1 + P_2 + \dots +P_k$ with
$P_1, P_2, \dots, P_k \in \mathcal P_1(\mathcal H)$. By Lemma \ref{lem 4.10} and Lemma \ref{lem 4.11}, $\Psi$ is well defined such that $r(\Psi(P))= r(P)$
for any $P \in \mathcal P(\mathcal H)$. Since $\phi$ shrinks the joint spectrum of any $n+1$ rank one projections, it also
shrinks the joint spectrum of any $n$ rank one projections. A similar argument as in the proofs of Lemma \ref{lem 4.2} and Lemma \ref{lem 4.3} can also imply
$r(\phi(Q_1) \vee \phi(Q_2) \vee \dots \vee \phi(Q_s)) \ge r(Q_1 \vee Q_2 \vee \dots \vee Q_s), \ \forall \ Q_1, Q_2, \dots, Q_s \in \mathcal P_1(\mathcal H)$. Hence
if $P = E_1 \vee E_2 \vee \dots \vee E_l$ with $E_1, E_2, \dots, E_l \in \mathcal P_1 (\mathcal H)$, then
$$\Psi(P) = \phi(E_1) \vee \phi(E_2) \vee \dots \vee \phi(E_l),$$
which implies that $\Psi(P \vee Q)= \Psi(P) \vee \Psi(Q)$. A similar argument as in the proof of Theorem \ref{main thm 2} gives that $\Psi$ is bijective and preserves the
joint spectrum of any two projections. The rest follows from the same line as in the proofs of Lemma \ref{lem 3.1} and Theorem \ref{thm 3.3}.

\begin{remark}
Notice that the method used in this subsection can not be applied to a surjective map on $\mathcal P_1(\mathcal H)$ which is joint spectrum shrinking for any $n$ rank one projections.
Indeed, even if a surjective map on $\mathcal P_1(\mathcal H)$ is joint spectrum preserving for any $n$ rank one projections, we can take $\beta_1, \beta_2, \dots, \beta_n$ in
Theorem \ref{main thm 2} to be non-orthonormal so that $\phi$ will not preserve the orthogonality. A further question is, for a surjective map on $\mathcal P_1(\mathcal H)$,
whether joint spectrum shrinking for any $n$ rank one projections implies joint spectrum preserving for any $n$ rank one projections.
\end{remark}

\end{document}